\documentclass[a4paper,11pt]{article}

\usepackage[english]{babel}
\usepackage[utf8]{inputenc}
\usepackage{paralist,url,verbatim,anysize}
\usepackage{amscd}
\usepackage{amsmath,amsthm,amssymb,amsfonts}
\usepackage{graphicx, graphics}
\usepackage{hyperref}
\theoremstyle{plain}
\newtheorem{theorem}{\bf Theorem}[section]
\newtheorem{conjecture}[theorem]{Conjecture}
\newtheorem{Cor}[theorem]{Corollary}
\newtheorem{lemma}[theorem]{Lemma}
\newtheorem{proposition}[theorem]{Proposition}

\newtheorem{question}[]{Question}
\newtheorem{example}[theorem]{Example}

\newtheorem{thmnonumber}{\bf Main Theorem}

\theoremstyle{definition}
\newtheorem{rem}[theorem]{Remark}

\DeclareMathAlphabet{\mathpzc}{OT1}{pzc}{m}{it}
\newenvironment{frcseries}{\fontfamily{frc}\selectfont}{}

\newcommand{\R}{\mathbb{R}}

\newcommand{\Lk}{\mathrm{link} \,}

\newcommand{\St}{\mathrm{star} \,}
\newcommand{\sd}{\mathrm{sd} \,}

\begin{document}

\author{Karim Adiprasito \thanks{Supported by DFG within the 
research training group ``Methods for Discrete Structures'' (GRK1408) and by Romanian
NASR, CNCS \& UEFISCDI, project number PN-II-ID-PCE-2011-3-0533}\\ \small
Inst.\ Mathematics, FU Berlin\\
\small \url{adiprasito@math.fu-berlin.de}
\and 
Bruno Benedetti \thanks{Supported by the G\"{o}ran Gustafsson Foundation and by the Swedish Research Council (Vetenskapsr{\aa}det) via the grant ``Triangulerade M{\aa}ngfalder, Knutteori i diskrete Morseteori''.} \\
\small
Dept.\ Mathematics, KTH Stockholm\\
\small \url{brunoben@kth.se}}

\date{\small April 7, 2012}
\title{Tight complexes in 3-space admit perfect discrete Morse functions}
\maketitle
\bfseries

\mdseries

\begin{abstract}
In 1967, Chillingworth proved that all convex simplicial $3$-balls are collapsible. 
Using the classical notion of tightness, we generalize this to arbitrary manifolds: We show that all tight simplicial $3$-manifolds admit some perfect discrete Morse function. 
We also strengthen Chillingworth's theorem by proving that all convex simplicial $3$-balls are non-evasive. In contrast, we show that many non-evasive $3$-balls are not convex.
\end{abstract}

\section{Introduction}
This paper explores the interplay of four properties: convexity, tightness, collapsibility and perfection. While the first two notions come from geometry, the last two are purely combinatorial. 

\textsc{Convexity} is a universally-known property of subsets of $\R^k$: A subset  of $\R^k$ is convex if it contains the whole segment between any two of its points. Since all simplicial complexes can be geometrically realized in some $\R^k$, one calls a simplicial complex ``convex'' if \emph{some} of its geometric realizations are convex. 
It turns out that all convex complexes are acyclic and contractible: They are in fact triangulations of topological balls.
 
\textsc{Tightness} is a way to extend the idea of convexity to non-acyclic complexes, studied among others by Kuiper~\cite{KuiperA,KuiperB} and K\"{u}hnel \cite{Kuehnel}. A subset $C$ of $\mathbb{R}^k$ is ``tight'' if any halfspace of $\mathbb{R}^k$ cuts out a subspace $C^+$ of $C$ such that the inclusion of $C^+$ into $C$ induces injective maps in homology. (We always consider homology with $\mathbb{Z}_2$-coefficients, but the results in this paper generalize to other fields of coefficients.) By convention, a complex is called ``tight'' if \emph{some} of its geometric realizations are tight. 

Convex balls are tight: The subspace cut out is always convex, and in particular contractible.  In the Thirties, Aumann proved a converse statement: Any acyclic tight complex is a convex ball \cite{Aumann}; see also K\"{u}hnel~\cite[Cor.~3.6]{Kuehnel}. In other words, a complex is convex if and only if it is tight and has trivial (reduced) homology. Tight manifolds need not be acyclic, since for any $g$ the genus-$g$ surface can be embedded tightly in $\R^3$.

\textsc{Collapsibility} is a combinatorial version of contractibility, introduced in 1939 by Whitehead \cite{Whitehead}. A \emph{free face} of a given complex is a face that belongs to only one other face. Not all complexes have free faces. A complex is ``collapsible'' if it can be reduced to a single vertex by repeteadly deleting a free face. All collapsible complexes are contractible, but the converse is false~\cite{BING}. The collapsibility property is of particular interest when applied to triangulations of manifolds. The authors have recently shown that collapsible manifolds are not necessarily balls~\cite{KarimBrunoMG&C}; however, collapsible PL manifolds are indeed balls~\cite[p.~293]{Whitehead}. 

\textsc{Perfection} is a property that extends the collapsibility notion to non-acyclic complexes. A \emph{discrete Morse function} on an arbitrary simplicial complex is a weakly-increasing map $f$, at most $2$-to-$1$, defined on the face poset of the complex and mapping into the poset of integers. The \emph{critical faces} of the complex under~$f$ are the faces at which $f$ is strictly increasing. A complex is called \emph{perfect} if it admits discrete Morse functions that are ``perfect'', namely, functions that for each $i$ have as many critical $i$-faces as the $i$-th Betti number of the complex. (We are considering homology with $Z_2$-coefficients, but there would be no trouble in defining ``$\mathbb{F}$-perfect complexes'', for different fields $\mathbb{F}$.)

Collapsible complexes are perfect: They admit a discrete Morse function with one critical vertex, and no further critical faces. In fact, a complex is collapsible if and only if it is acyclic and perfect~\cite[Theorem~3.3]{FormanADV}. Perfect complexes need not be acyclic, since for any genus $g$, any triangulation of the genus-$g$-surface is perfect.

\medskip
A non-trivial relation between the previous properties was noticed in 1967 by Chillingworth~\cite{CHIL}:

\begin{theorem}[Chillingworth] \label{thm:Chillingworth}
Every convex $3$-complex is collapsible.
\end{theorem}

\noindent Lickorish's question of whether all convex $4$-complexes are collapsible, is to the present day unsolved. 
In general, finding discrete Morse functions with as few critical cells as possible by hand (or by computer) is a difficult task, with a priori no guarantee of success. This is why results like Theorem~\ref{thm:Chillingworth} are important also from the perspective of computational geometry.

In the present paper, we generalize Theorem~\ref{thm:Chillingworth} to complexes that are not necessarily acyclic, by constructing perfect discrete Morse functions on them.

\begin{thmnonumber}[Theorems~\ref{thm:tomo3ball} \& \ref{thm:tomo3man}] \label{mainthm:chilliTIGHT}
Every tight complex in $\R^3$ is perfect. Every $3$-manifold that has a tight embedding in some $\R^k$ is perfect.
\end{thmnonumber}

Note that for acyclic complexes ``tight'' boils down to ``convex'' and ``perfect'' boils down to ``collapsible'', so that Main Theorem~\ref{mainthm:chilliTIGHT} indeed reduces to Theorem~\ref{thm:Chillingworth}.

In 1984, Kahn, Saks and Sturtevant characterized for simplicial complexes the property of \textsc{evasiveness}, originally introduced by Karp in the context of theoretical computer science. They showed that {non-evasive strictly implies collapsible. Later Welker proved that the barycentric subdivision of every collapsible complex is non-evasive~\cite[Theorem~2.10]{Welker}. 
It turns out that one can improve Theorem~\ref{thm:Chillingworth} by strengthening its conclusion: 

\begin{thmnonumber}[Corollary~\ref{cor:chillingworth}] \label{mainthm:chilliNE}
Every convex $3$-complex is non-evasive. 
\end{thmnonumber}

For example, Rudin's ball, as well as any linear triangulation of the $3$-simplex, is non-evasive. That said, not all collapsible balls are evasive, as recently shown by Benedetti--Lutz~\cite{ BenedettiOWR, BenedettiLutz}. We proved in~\cite[Theorem~3.42]{KarimBrunoMG&C} that every convex $d$-complex becomes non-evasive after at most $d-2$ barycentric subdivisions. 
When $d=3$, Main Theorem~\ref{mainthm:chilliNE} improves this bound by one unit. 

The converse of Main Theorem~\ref{mainthm:chilliNE} does not hold: Some $3$-balls are non-evasive, without even admitting linear embeddings in $\mathbb{R}^3$. To prove this, we have to take a detour in classical knot theory, and extend results by Lickorish and Martin to knots of bridge index three:

\begin{thmnonumber}[Proposition~\ref{prop:LickorishClaim}] \label{mainthm:knots}
Every composite knot of bridge index $3$ appears as $3$-edge subcomplex in a suitable, \emph{perfect} triangulation of $S^3$. 
\end{thmnonumber}
 
In the last part of the paper, we discuss what happens in higher dimensions. Whether all convex $d$-balls are collapsible, remains an open question~\cite[Problem~5.5]{Kirby}. However, we can construct $4$-balls that are tight with respect to {\em some} directions, and which are far from being collapsible ({\bf Theorem~\ref{thm:tomo4ball}}).

\section{Main Results}

In this section, we strengthen Chillingworth's theorem (``all convex $3$-balls are collapsible'') in two directions: First we show that convex $3$-balls are non-evasive, which is stronger than being collapsible ({\bf Corollary~\ref{cor:chillingworth}}). Then we show that all tight complexes in $\R^3$ admit perfect discrete Morse functions ({\bf Theorem~\ref{thm:tomo3ball}}). This statement boils down to Chillingworth's theorem in case the reduced homology of the complex is trivial. 

We denote the $i$-th (non-reduced) Betti number by $\beta_i$. If $\sigma$ is a face of simplicial complex $C$, by $C - \sigma$ we denote the {\it deletion of $\sigma$ from $C$}, which is the subcomplex of $C$ given by all faces of $C$ that do not contain $\sigma$. 

Let $|C|$ be  a geometric realization in $\R^k$ of a simplicial complex $C$. 
Let $\pi$ be a nonzero vector in $\R^k$. 
If $h$ is any affine hyperplane orthogonal to $\pi$, we denote by $h^+$ (resp. $h^-$) the halfspace delimited by $h$ in direction $\pi$ (resp. in direction $-\pi$).
We say that $|C|$ is {\it $\pi$-tight} if for  every hyperplane $h$ orthogonal to $\pi$ the homomorphism 
\[ H_i ( h^+ \cap |C|) \longrightarrow H_i (|C|) \]
induced by the inclusion $h^+ \cap |C| \subset |C|$ is injective. 
We say that $|C|$ is {\em tight} if $|C|$ is $\pi$-tight for almost all vectors $\pi$. ``Almost all'' means that we allow a measure-zero set of exceptions;  we will always assume $\pi \ne 0$.

This definition of tightness can be found for example in K\"{u}hnel~\cite{Kuehnel} or Effenberger~\cite{Effenberger}. We should mention that a slightly different definition was used by Kuiper to study smooth embeddings of manifolds, cf.~\cite{KuiperA}.
The notion of tightness can be viewed as a generalization of the notion of convexity to arbitrary topological types: 

\begin{theorem}[Aumann \cite{Aumann}]\label{thm:Aumann}
Let $|C|$ be  a linear embedding of a simplicial complex $C$ into $\R^k$.  $|C|$ is convex if and only if it $|C|$ is tight and has trivial homology.
\end{theorem}

We study tight complexes in $\R^k$ by reducing them to complexes that embed in $\R^{k-1}$. This is particularly effective when $k=3$, since planar complexes have all sorts of nice properties. The reduction step consists in intersecting any $\pi$-tight embedding of the complex with a hyperplane orthogonal to $\pi$. This is explained in the following Lemma.

\begin{lemma} \label{lem:acyclic}
Let $|C|$ be a $\pi$-tight linear embedding of a simplicial complex $C$ in $\R^k$, such that $\pi$ is in general position with respect to $C$ (i.e. no two vertices have the same value with respect to $\langle \pi , - \rangle$. 
Let $v$ denote the vertex of of $C$ that maximizes the inner product $\langle \pi , - \rangle$. Then
\begin{compactenum}[\rm (1)]
\item the complex $C- v$ has a $\pi$-tight embedding in $\R^k$, and
\item $\beta_i \, (\Lk(v,C)) + \beta_{i+1} ( C-v ) -\delta_{i0}=  \beta_{i+1} (C)\, $,	 where $\delta_{ij}$ is the Kronecker delta.
\end{compactenum}
\end{lemma}

\begin{proof} With respect to the hyperplane $h$, we denote by $h^-$ the open halfspace containing the vertex $v$, and by $h^+$ the open halfspace containing all the other vertices. Note that $h^+ \cap |C|$ is combinatorially equivalent to $C-v$, and that $h \cap |C|$ is combinatorially equivalent to $\Lk(v,C)$: See Figure~\ref{fig:tight}.

\begin{figure}[htbf]
	\centering
\includegraphics[width=.23\linewidth]{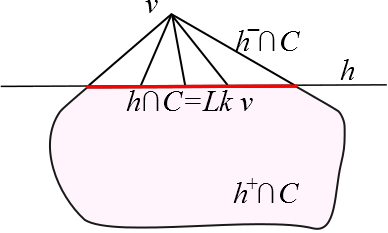}
\vskip-2mm
    \caption{\footnotesize If the hyperplane $h$ is close enough to the vertex $v$, the halfspace $h^+$ delimited by $h$ and not containing~$v$ intersects the complex in a space that is  homotopy equivalent to $C- v$, the deletion of $v$ from $C$. The intersection of $h$ with $C$ is combinatorially equivalent to $\Lk (v,C)$.}
    \label{fig:tight}
\end{figure}

Consider the following long exact Mayer-Vietoris sequence:
\[ 
\!\!\!\!\!\! \ldots \; \rightarrow 
\, \tilde{H}_{i+1}(C) \, \rightarrow \, 
\tilde{H}_i( h \cap |C|) \, \rightarrow \, 
\tilde{H}_i(h^+\cap |C|)\oplus \tilde{H}_i(h^-\cap |C|) \, \rightarrow \, 
\tilde{H}_i(C) \, \rightarrow \; \ldots \; \rightarrow 0. 
\] 
where $\tilde{H}_\ast$ denotes reduced homology. $\tilde{H}_i(h^-\cap |C|)$ is combinatorially equivalent to $\St(v,C)$, which is contractible. By the definition of tightness, the map $\tilde{H}_i(h^+\cap |C|)\oplus \tilde{H}_i(h^-\cap |C|) \rightarrow \tilde{H}_i ( C )$ is injective, whence we obtain the short exact sequence
\[ 
0 \;\; \rightarrow \;\;
\tilde{H}_{i+1}(h^+\cap |C|)\oplus \tilde{H}_{i+1}(h^-\cap |C|) 
\;\; \rightarrow \;\; 
\tilde{H}_{i+1} (|C|)
\;\; \rightarrow \;\; 
\tilde{H}_i(h\cap |C|) 
\;\; \rightarrow \; 
0. 
\] 
This implies item (2) immediately, since $h \cap |C| = \Lk(v,C)$ and $h^+ \cap |C|$ is isomorphic to $C-_s v$. To prove (1), we choose for $h^+ \cap |C|$ its natural embedding inside $|C|$. Let $z$ be an arbitrary hyperplane orthogonal to $\pi$, and consider the open halfspace $z^+$ delimited by $z$ in the direction of $\pi$. We want to prove that the inclusion $z^+ \cap ( h^+ \cap |C|) \hookrightarrow ( h^+ \cap |C| )$ induces injective maps in homology. By the $\pi$-tightness of $C$, we already know that $z^+ \cap |C| \hookrightarrow |C|$ induces injective maps in homology. Moreover, reasoning as above, we obtain two exact sequences connecting the homology of $h \cap |C|$ to the homologies of the intersection of halfspaces with $|C|$ and with $h^+ \cap |C|$, respectively. Hence we have the following commutative diagram:

\[
 \begin{CD}
0 
@>>>  
\; H_i ( h^+ \cap C) \;
@>>>
\; H_i (C) \;
@>>>
\; H_{i-1} (h \cap |C|) \;
@>>>
0 \\
@.
@A{j}AA
@A{i}AA
@|
@.
\\
0 
@>>>  
\; H_i ( h^+ \cap z^+ \cap C) \;
@>>>
\; H_i (z^+ \cap C) \;
@>>>
\; H_{i-1} ( h \cap |C| ) \;
@>>>
0 \\
 \end{CD}
\]
By the Snake Lemma, since the map $i$ is injective, so is $j$.
\end{proof}

Recall that a $0$-dimensional complex is \emph{non-evasive} if and only if is a point. Recursively, a $d$-dimensional simplicial complex ($d>0$) is called \emph{non-evasive} if and only if there is some vertex $M$ whose link and deletion are both non-evasive. A \emph{planar} complex is a complex that embeds in $\mathbb{R}^2$, or equivalently, that is combinatorially equivalent to the subcomplex of some triangulated disc. According to this definition, any triangulation of $S^2$ is \emph{not} planar.

\begin{lemma} \label{lem:planar}
Every planar complex admits a perfect discrete Morse function, and every planar complex with trivial first homology is a disjoint union of non-evasive complexes.
\end{lemma}

\begin{proof}
Without loss of generality, we can assume that the complex is connected. Every graph admits a perfect discrete Morse function, and every tree is non-evasive; so if $\dim D < 2$ the claim is clear. Suppose that $D$ is a $2$-dimensional connected complex, and embed it linearly in the plane. Since $D$ is planar, every edge of $D$ is contained in at most two triangles, and by maximizing some linear functional over $|D|$ we can find at least one edge $e$ that is contained in exactly one triangle. So $D$ collapses onto the complex $D' = D - T - e$, which is also planar. By induction on the number of facets, the complex $D'$ admits a perfect discrete Morse function. So also $D$ does. 

Now, assume $H_1 (D) = 0$. Since $D$ is planar and connected, $H_0(D) = \mathbb{Z}$ and $H_2(D)=0$. Since $D$ admits a perfect discrete Morse function, $D$ is collapsible. Write $D$ as the union of some pure $2$-complexes $D_i$ with some trees $T_i$, so that the intersection of two different $D_i$'s is either a single point, or empty. 
Any of the trees $T_i$ that does not connect two or more $D_i$ is clearly nonevasive, and can be removed without loss of generality. Similarly, every $D_i$ that is connected to the rest of the complex via a unique vertex is nonevasive, and can  be removed. Since $D$ contains no cycles, and is homotopic to a tree, there is always either a tree $T_i$ that is connected to only one $D_i$, or a $2$-complex $D_i$ that is connected to the rest via a single vertex.
\end{proof}

\begin{rem}\rm
Lemma~\ref{lem:planar} says that while planar complexes can be successfully simplified via elementary collapses, \emph{acyclic} planar complexes can be simplified even faster, by deleting vertices whose link is a path. Can all planar complexes be simplified faster? The answer is negative. In fact, consider the planar $1$-complex~$E$ of Figure~\ref{fig:wedge}, consisting of the four triangles $[1,2,3]$, $[3,4,5]$, $[1,5,6]$ and $[2,4,6]$. Topologically, $E$ retracts to a bouquet of three $1$-spheres. If we want to perform an elementary collapse on $E$, all edges of $E$ are ``free'' (that is, they all belong to one triangle only.) However, if we want to perform a non-evasiveness step on $E$, all vertices of $E$ are ``blocked'': Every vertex link consists of two disjoint edges. 
\end{rem}

\begin{figure}[htbf]
	\centering
\includegraphics[width=.14\linewidth]{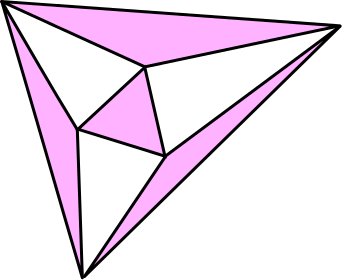}
    \caption{\footnotesize A planar complex $E$ whose vertex links are all disconnected (hence evasive).}
    \label{fig:wedge}
\end{figure}

Reasoning as in the proof of Lemma~\ref{lem:planar}, one can reach the following Lemma, whose proof is left to the reader:

\begin{lemma}\label{lem:planar2}
If $D$ is a subcomplex of a planar complex $C$ and $|D|$ is a deformation retract of $|C|$, then $C$ collapses to $D$.
\end{lemma}

\begin{theorem}\label{thm:tomo3ball}
Let $C$ be a simplicial complex that admits a $\pi$-tight linear embedding in $\mathbb{R}^3$, where $\pi$ is in general position with respect to $C$, as in Lemma \ref{lem:acyclic}. Then $C$ admits a perfect discrete Morse function. If in addition $C$ has trivial homology, then $C$ is also non-evasive.
\end{theorem}

\begin{proof}
We proceed by induction on the number of vertices. Let $v$ be the highest vertex in the direction of~$\pi$. By induction, $C- v$ admits a perfect Morse matching $f$. The complex $\Lk(v,C)$ is planar, so by Lemma \ref{lem:planar} it admits a perfect Morse matching $f'$. This lifts to a Morse matching on $\St(v,C)$ 
as follows: if $(\sigma,\Sigma)$ is a matching pair in $f'$, we match $v\ast\sigma$ with $v\ast\Sigma$, and if $w$ is some critical vertex of $f'$, we match $v$ with $v \ast w$. Let $f''$ be the resulting Morse matching on $\St(v,C)$. The union of the matchings induced by $f$ and $f''$ yields a Morse Matching on $C$. Let $g$ be any discrete Morse function that induces the latter matching. By construction, for all $i$ we have $c_{i+1} (g)=c_{i+1} (f) + c_{i+1} (f'')=c_{i+1} (f) + c_{i} (f')-\delta_{i0}$. Now, $c_i (f') = \beta_i (\Lk(v,C))$ and $c_{i+1} (f)= \beta_{i+1} (C - v)$. By Lemma~\ref{lem:acyclic}(ii),  $c_{i+1} (g)= \beta_{i+1} (C)$. Hence the function $g$ is perfect.

If $C$ has trivial homology, by Lemma~\ref{lem:acyclic}(ii) $\Lk(v,C)$ has trivial reduced homology groups. So we can apply Lemma \ref{lem:planar} and conclude that $\Lk(v,C)$ is non-evasive. Thus, removing $v$ is a valid nonevasiveness step. Since, $C- v$ is $\pi$-tight by Lemma \ref{lem:acyclic}(i), and has less vertices than $C$, the complex  $C- v$ is non-evasive by induction. This completes the proof.
\end{proof}

With the same argument, one can reach the following statement:

\begin{theorem}\label{thm:tomo3man}
Let $C$ be a triangulation of a $3$-manifold with a $\pi$-tight embedding in some $\R^k$, such that $\pi$ is in general position w.r.t. $C$. Then $C$ admits a perfect discrete Morse function.
\end{theorem}

\begin{Cor}[Chillingworth] \label{cor:chillingworth}
Let $B$ be a $3$-ball that admits a convex linear embedding in $\mathbb{R}^3$. Then $B$ is non-evasive and collapsible.
\end{Cor}

\begin{proof}
A convex $3$-ball is $\pi$-tight for any $\pi$: In fact, the intersection of any open halfspace with a convex ball is still convex and thus contractible; compare Theorem~\ref{thm:Aumann}. By Theorem~\ref{thm:tomo3ball}, $B$ is non-evasive, so in particular it is collapsible. 
\end{proof}

\begin{example} \em Rudin's $3$-ball is a non-shellable subdivision of a tetrahedron in $\mathbb{R}^3$. Since the subdivision is linear, it is non-evasive. An explicit sequence of vertices that proves non-evasiveness of Rudin's ball has been recently found in \cite{BenedettiLutz}.
\end{example}

\begin{proposition}[Nested polytopes]\label{pro:2polyt}
Let $P,\ Q$ be convex simplicial polytopes in $\R^3$ with $Q \subset P$. Suppose that $\partial P$ and $\partial Q$ intersect in at most $2$ vertices. 
Let $C$ be a linear subdivision of the space $|P| \setminus \operatorname{int} |Q|$, such that the triangulation of $C$ restricted to $\partial P$ resp. $\partial Q$ coincides with the natural triangulation of $\partial P$ resp. $\partial Q$. For any facet $\sigma$ of $\partial P$, we have $C \, \searrow (\partial P - \sigma) \cup \partial Q$. 
\end{proposition}

\begin{proof}
Since $\partial P$ and $\partial Q$ intersect in at most $2$ vertices, $\sigma$ is not in $\partial Q$. The space $|P| \setminus \operatorname{int} |Q|$ is tight, so $C$ is embedded tightly. Choose a direction $\pi$ such that a certain vertex of $\sigma$ maximizes $\langle \pi, - \rangle$, and assume that such vertex is not in $Q$. We proceed as in Theorem \ref{thm:tomo3ball}, recursively deleting the top vertex of $C$ in the order induced by the function $\langle \pi, - \rangle$. In fact, at each step, let us denote by $C$ resp.~$C'$ the complex before resp. after the deletion of the top vertex. If $C$ is $\pi$-tight, obviously $C'$ is still $\pi$-tight. Let us look at the highest vertex $v$ of $C'$ (i.e. the vertex maximizing  $\langle \pi, - \rangle$ on $C'$) and study $\Lk (v, C')$. There are four cases:
\begin{compactenum}[(i)]
\item If $v$ is the top vertex of $|P| \setminus \operatorname{int} |Q|$, then $\Lk(v,C')=\Lk(v,C)$ is a 2-ball and $\Lk(v,\partial P) = \Lk(v,\partial P) \, \cap \, \Lk(v,C')$ is its boundary. After the deletion of the free face $\Lk(v,\sigma)$, $\Lk(v,C')$ collapses onto $\Lk(v,\partial P)$ by Lemma \ref{lem:planar2}.
\item If $v$ is in $\partial P$ but not in $\partial Q$, then $\Lk(v,C')$ is a triangulation of a contractible planar complex, of which $\Lk(v,\partial P)\cap \Lk(v,C')$ is a contractible subcomplex. By Lemma \ref{lem:planar2}, the former collapses onto the latter.
\item If $v$ is in $\partial Q$ but not in $\partial P$, then $\Lk(v,C')$ is a 2-ball, of which $\Lk(v,\partial Q) \, \cap \, \Lk(v,C')$ is a contractible subcomplex. By Lemma \ref{lem:planar2}, the former collapses to the latter. 
\item If $v$ is in $\partial Q \cap \partial P$, then $\Lk(v,C')$ is a contractible planar complex, of which $(\Lk(v,\partial Q)\cup \Lk(v,\partial P) ) \cap \Lk(v,C')$ is a contractible subcomplex. By Lemma \ref{lem:planar2}, we can collapse the former onto the latter. 
\end{compactenum}
\vskip-5mm
\end{proof}

\enlargethispage{4mm}
\vskip-1mm
\section{Non-evasive balls without convex realizations} \label{sec:knots}
In this section, we show that the converses of Corollary~\ref{cor:chillingworth} and Theorem~\ref{thm:tomo3ball} do not hold. In fact, we show that a $3$-ball can be non-evasive and collapsible, without admitting any $\pi$-tight realization in $\mathbb{R}^3$ {\bf (Theorem~\ref{thm:PiTightNonCollapsible})}. In general, \emph{proving} that a certain realization is $\pi$-tight, is a relatively easy task; in contrast, \emph{excluding} the existence of $\pi$-tight realizations for some non-evasive complex, seems a much harder problem. To deal with it, we have to recall and expand some classical results in knot theory. 

Let us start by recalling the basics. For the definitions of knot, diagram and bridge index, we refer the reader to the textbook by Kawauchi~\cite[pp.~7--18]{KAWA}. All the knots we consider are \emph{tame}, that is, realizable as $1$-dimensional subcomplexes of some triangulated $3$-sphere. A \emph{connected sum} of two given knots is obtained by cutting out a tiny arc from each, and then by sewing the resulting curves together along the boundary of the cutouts. The bridge index of a connected sum of two knots equals the sum of their bridge indices, minus one. The \emph{trivial} knot is one that bounds a disc. A knot is called \emph{composite} if it is the connected sum of two non-trivial knots, and \emph{prime} otherwise. 

A \emph{$K$-knotted
spanning arc} of a $3$-ball $B$ is a path $P$ of consecutive interior
edges of~$B$, such that the two endpoints of~$P$ lie on the
boundary $\partial B$, and any path on $\partial B$ between these endpoints
completes $P$ to the knot $K$. (Since $S^2$ is simply-connected, the knot type does not depend on the boundary path chosen.) Note that the relative interior of the arc is allowed to intersect the boundary of the $3$-ball~\cite{EH}. If the arc consists of one edge only, we call it \emph{$K$-knotted spanning edge}.

The minimal number of edges needed to realize a given knot \emph{as a polygonal path in $\mathbb{R}^3$} is known as \textsc{stick number} of the knot. The trivial knot has stick number $3$; the trefoil knot and its mirror image have stick number $6$; all other knots have stick number $\ge 7$~\cite[Theorem~4]{Calvo}. The stick number can be arbitrarily large: For example, the connected sum of $t$ trefoils has stick number $2t + 4$~\cite{Rolfsen}. Even knots of fixed bridge number can have arbitrarily large stick number: For example, the knot of Figure \ref{fig:1039} has bridge index $2$ and stick number $13$. 

\medskip
\begin{figure}[htbf]
	\centering
\includegraphics[width=.27\linewidth]{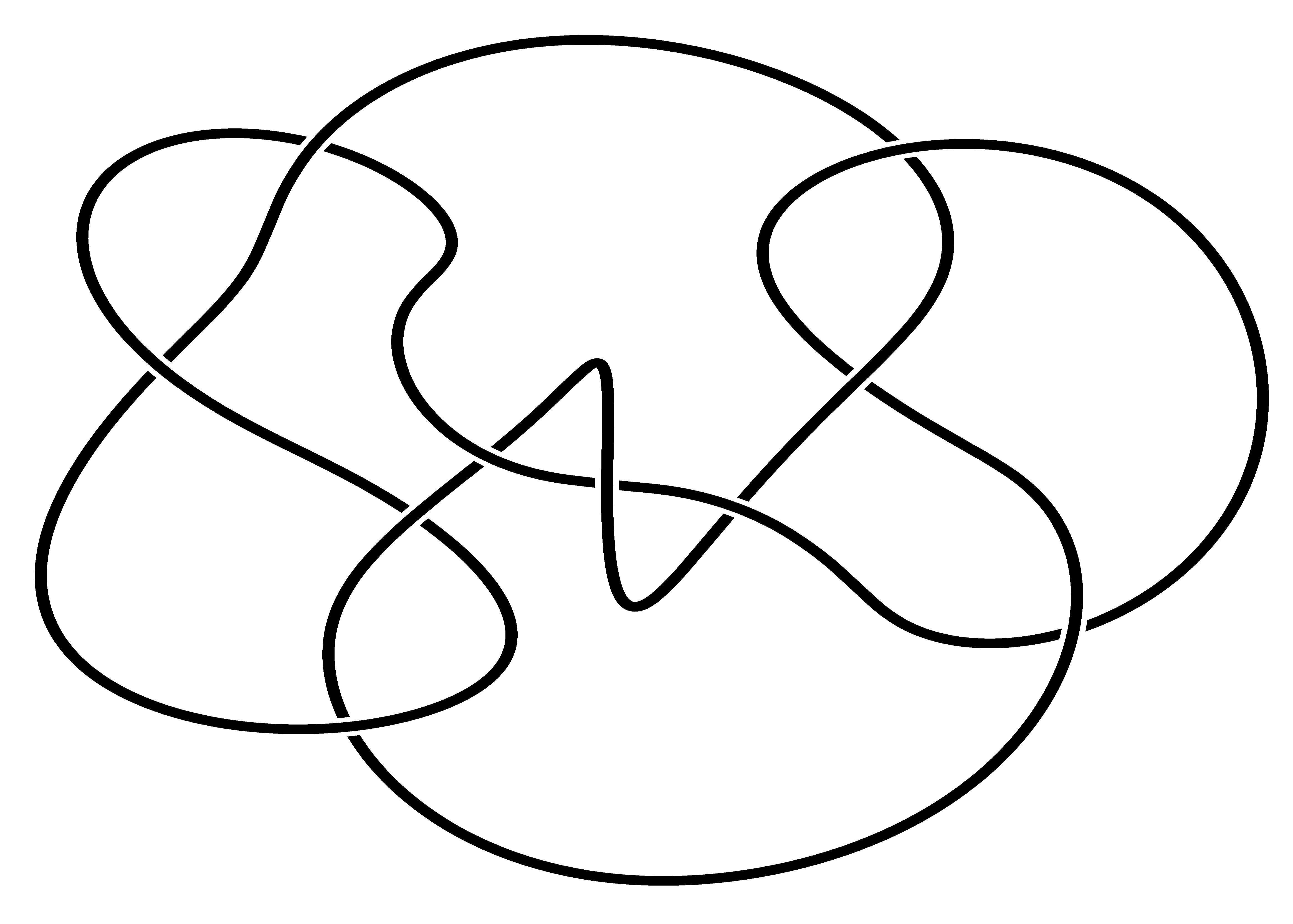}
\caption{\footnotesize The knot $10_{39}$ from Rolfsen's table~\cite{Rolfsen}.}
    \label{fig:1039}
\end{figure}

In contrast, the minimal number of edges needed to realize a given knot \emph{in some triangulated $3$-sphere or $3$-ball} does not depend on the knot, and cannot be arbitrarily large. In fact, this number is always $3$. This was known already in the Sixties, by the work of Bing and others: 

\begin{proposition}[{Furch~\cite[p.~73]{Furch}, Bing~\cite[p.~110]{BING}}] \label{prop:Bing}
Any knot can be realized as knotted spanning edge in some triangulated $3$-ball, linearly embedded in $\mathbb{R}^3$. 
\end{proposition}

\begin{proof}
Let us take a finely triangulated $3$-ball in $\mathbb{R}^3$ and drill a tubular hole in it, from the top to the bottom, along the chosen knot $K$. If we stop
one step before ``perforating'' the $3$-ball completely (that is, if we stop when
we are at a distance of one edge $[x,y]$ from the bottom), we obtain a $3$-ball $B$
with a $K$-knotted spanning edge. In fact, any path in $\partial B$ from $x$ to $y$ must ``climb up'' the tubular hole, and thus it closes up $[x,y]$ to a knot isotopic to $K$. 
\end{proof}

\begin{Cor} \label{cor:bing}
Any knot can be realized as $3$-edge subcomplex of some triangulated $3$-sphere (or of some triangulated $3$-ball, which is not linearly embeddable in $\R^3$ if the knot is non-trivial).
\end{Cor}

\begin{proof}
Let $B$ be a $3$-ball with a $K$-knotted spanning edge $[x,y]$. Let $v$ be a new vertex. The $3$-sphere $S_B := \partial (v \ast B)$ contains a $3$-edge subcomplex $[x,y] \cup [y,v] \cup [y,v]$ isotopic to $K$. Removing any tetrahedron from $S_B$ one gets a $3$-ball with the same knotted subcomplex. 
\end{proof}

In 1969, Lickorish and Martin obtained a more sophisticated result. 

\begin{proposition}[{Lickorish--Martin \cite{LICKMAR}; Hamstrom--Jerrard~\cite{HamstromJerrard}; see also Benedetti--Ziegler~\cite[Theorem~3.23]{BZ}}] \label{prop:LickMar}
Any knot of bridge index~$2$ can be realized as knotted spanning edge in some \emph{collapsible} $3$-ball (which might not be  linearly embeddable in $\R^3$).
\end{proposition}

\begin{Cor}[Lickorish~\cite{LICK}] \label{cor:lickmar}
Any knot of bridge index~$\le 2$ can be realized as $3$-edge subcomplex of some \emph{perfect} $3$-sphere (or $3$-ball).
\end{Cor}

\begin{proof}
If $K$ is the trivial knot, then $K$ appears as $3$-edge subcomplex of the boundary of the $4$-simplex, and the claim is obvious. If $K$ is a knot of bridge index~$2$, let $B$ be a collapsible $3$-ball with a $K$-knotted spanning edge. Construct $S_B = \partial (v \ast B)$ as in Corollary~\ref{cor:bing}.  If $\sigma$ is a facet of $\partial B$, the complex $\partial B - \sigma$ is a $2$-ball, and thus it is collapsible onto a vertex $w$. This implies that $v \ast \partial \sigma - v \ast \sigma$ collapses onto $\partial \sigma \cup (w \ast v)$. The same collapsing sequence shows also that $S_B - v \ast \sigma$ collapses onto $B \cup (w \ast v)$. But $B \cup (w \ast v)$ collapses onto $B$, which is collapsible. So, the removal of the facet $v \ast \sigma$ from $S_B$ yields a collapsible $3$-ball. This $3$-ball contains the knot $K$ as $3$-edge subcomplex. 
\end{proof}








In Proposition~\ref{prop:LickMar}, the bound on the bridge index is best possible. In fact, let $T_t$ be a connected sum of $t$ trefoils. The knot $T_t$ has bridge index $t+1$. By the work of Goodrick~\cite{GOO}, if $T_t$ appears as knotted spanning edge in a triangulated $3$-ball $B$, and $t \ge 2$, then $B$ cannot be collapsible.

Surprisingly, it turns out that the bound on the bridge index in Corollary~\ref{cor:lickmar} is \emph{not} best possible. We are going to show here that many knots of bridge index $3$, including the connected sums of two trefoil knots, live as $3$-edge subcomplexes inside some perfect triangulation of the $3$-sphere.

\begin{lemma}\label{lem:projective}
Let $Q$ be an arbitrary $d$-polytope. Let $v, w$ be two distinct vertices of $Q$. There is a polytope $Q'$ combinatorially equivalent to $Q$ and embedded in $\R^d$ so that the vertex of $Q'$ corresponding to $v$ maximizes the ``quota function'' $x_d$ on $Q'$, and the vertex corresponding to $w$ minimizes $x_d$ on $Q'$.
\end{lemma}

\begin{proof}
Let $H_v$ (resp. $H_w$) be a hyperplane tangent to $Q$ in $v$ (resp.~in $w$). Via a suitable projective transformation $\pi$, we can move the intersection of $H_v \cap H_w$ to the hyperplane at infinity, and make $\pi (H_v)$ and $\pi(H_w)$ parallel to $x_d=0$ in $\mathbb{R}^d$. Up to performing a further reflection, the image of $Q$ under the projective transformation $\pi$ is the polytope $Q'$ desired.
\end{proof}

\begin{lemma}\label{lem:collapsepyramid}
Let $A$ be the boundary of an $n$-gon $D$. For any simplicial $3$-polytope $Q$ and for any two vertices $v, w$ of $Q$, the ball $P = \operatorname{susp}(D)$ admits a decomposition into polytopes that
\begin{compactenum}[\em (1)]
\item contains a subcomplex $Q'$ combinatorially equivalent to $Q$,
such that $v$ and $w$ are identified with the apices of the suspension, and $Q' \cap \partial P = \{ v \} \cup \{ w \}$;
\item coincides with the triangulation $\operatorname{susp}(A)$ when restricted to the boundary of $P$;
\item collapses onto $Q' \, \cup \, (A \ast w)$.
\end{compactenum}
\end{lemma} 

\begin{proof}
Let us embed the $1$-sphere $A$ in the plane $z=0$ of $\mathbb{R}^3$, as the boundary of an $n$-gon around the origin. By Lemma~\ref{lem:projective}, $Q$ is combinatorially equivalent to some $3$-polytope $Q' \subset \mathbb{R}^3$, such that the vertex $v$ (resp. $w$) is identified with the strictly highest (resp. lowest) point of $Q'$. Up to a further translation, we can assume that the lowest point $w'$ of $Q'$ is the origin. Up to wiggling $Q'$ (or $A$) a little, we can also assume that no vertex of $A$ belongs to the affine hull of a facet of $Q'$. Let us call $v'$ the highest point of $Q'$. Let us construct the pyramid $P$ with basis $w' \ast A$ and apex~$v'$. If $A$ is chosen large enough, $Q'$ touches $P$ only at the origin and at the apex $v'$: See Figure \ref{fig:LostResult}.

\begin{figure}[htbf]
  \centering 
  \includegraphics[width=0.3\linewidth]{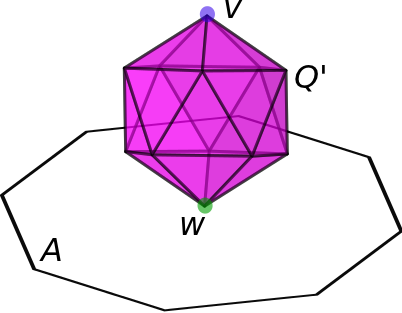} 
\caption{\small Given a polytope $Q'$ (purple) in $\mathbb{R}^3$, call $v$ its top vertex (blue) and $w$ its bottom vertex (green). Take a $1$-sphere $A \subset \mathbb{R}^3$ so that $w$ is coplanar with $A$. If $A$ is chosen `large enough' (larger than in this figure), then the pyramid $|v \ast A \ast w |$ (which is the convex hull of $v \cup A$) contains $Q'$ and admits a triangulation that collapses onto $Q' \cup (A \ast w)$.}
  \label{fig:LostResult}
\end{figure}

Now, any triangulation of the space~$|P| \setminus \operatorname{int} |Q'|$ that satisfies conditions (1) and (2), satisfies also condition (3), because of Corollary \ref{pro:2polyt}. 
One way to triangulate $|P| \setminus \operatorname{int} |Q'|$ so that conditions (1) and (2) are satisfied, is the following. We choose a vertex $p$ of $A$, and attach to $Q'$ the cone with apex $p$ over the faces of $\partial Q'$ seen by $p$. Let $Q''$ be the obtained complex, and let $p'$ be the vertex of $A$ following $p$ in the clockwise order. We attach to $Q''$ the cone with apex $p'$ over the faces of $\partial Q'$ seen by $p'$. And so on, for all vertices of $A$, in clockwise order.
\end{proof}

\begin{proposition} \label{prop:LickorishClaim}
Any composite knot of bridge index $3$ can be realized as $3$-edge subcomplex in a perfect triangulation of a $3$-sphere (or $3$-ball).
\end{proposition}

\begin{proof}
Being composite, $K$ is the connected sum of two knots $K_1$ and $K_2$ of bridge index~$2$. By Proposition~\ref{prop:LickMar}, we can construct two collapsible $3$-balls $B_1$ and $B_2$, so that each $B_i$ contains a knotted spanning edge $[x_i, y_i]$, the knot being isotopic to $K_i$. 

Let $A:= \Lk (y_1, \partial B_1 )$. 
Let $P$ be the pyramid $|y_1 \ast A \ast y_2 |$. Since all $2$-spheres are polytopal, there is a simplicial $3$-polytope $Q$ with $\partial Q = \partial B_2$. Applying Lemma \ref{lem:collapsepyramid}, and later replacing the  triangulation of~$Q$ with $B_2$, we obtain that there is a new triangulation $\tilde{P}$ of $P$ such that
\begin{compactenum}[(1)]
\item $\tilde{P}$ contains a copy of $B_2$, so that the $K_2$-knotted spanning edge of $B_2$ goes from $y_1$ to $y_2$;
\item $\tilde{P}$ collapses onto $B_2 \cup (y_1 \ast A)$; 
\item on the boundary of $P$, the two triangulations $\tilde{P}$ and $y_1 \ast A \ast y_2$ coincide. 
\end{compactenum}

By Lemma~\ref{lem:planar2}, any $2$-ball collapses onto the closed star of any of its vertices. In particular, if $\sigma$ is a facet of $\partial B_1$, the $2$-ball $\partial B_1 - \sigma$ collapses onto $y_1 \ast A$. This implies that $(\partial B_1 \ast y_2) - (\sigma \ast y_2)$ collapses onto $y_1 \ast A \ast y_2$. By attaching $B_1$ and observing that this attachment does not interfere with the collapse, we obtain 
\[S  - (\sigma \ast y_2) \ \searrow \ (y_1 \ast A \ast y_2) \; \cup \; B_1,\]
where $S := \partial (B_1 \ast y_2)$. The collapse above does not depend on how the interior of $y_1 \ast A \ast y_2$ is triangulated. Since $\tilde{P}$ is an ``alternative'' triangulation of the interior of $y_1 \ast A \ast y_2$, but it coincides with it on the boundary, we get
\[\tilde{S}  - (\sigma \ast y_2) \ \searrow \ \tilde{P} \; \cup B_1,\]
where $\tilde{S}  := B_1  \, \cup \,  \tilde{P} \, \cup \, (\textrm{del} (x, \partial B_1) \ast y_2 )$ is the triangulation of $S^3$ ``alternative'' to $S$, and obtained by re-triangulating $|y_1 \ast A \ast y_2|$ according to $\tilde{P}$. But by construction $\tilde{P}$ collapses onto $B_2 \cup (y_1 \ast A)$, and $y_1 \ast A$ is a subcomplex of $B_1$. Therefore,
\[\tilde{S}  - (\sigma \ast y_2) \ \searrow \ B_2 \, \cup_{y_1} B_1,\]
where $\cup_{y_1}$ denotes the wedge at $y_1$. 
Since $B_1$ and $B_2$ are collapsible, their wedge is also collapsible. In other words, the removal of the facet $\sigma \ast y_2$ from $\tilde{S}$ yields a collapsible $3$-ball. In this ball, the knotted spanning edges are concatenated, so the knots ``sum up''. But by construction, the connected sum of the knots $K_1$ and $K_2$ is the knot $K$ we started with. \end{proof}

\begin{lemma} \label{lem:glueNE}
Let $W$ be the wedge of two non-evasive $3$-balls. There is a non-evasive $3$-ball $B$ that contains $W$ and can be reduced to $W$ by deleting a single boundary vertex of $B$.
\end{lemma}

\begin{proof}
Let $B_1$ and $B_2$ be two non-evasive $3$-balls. Choose triangles $\{a_i, b_i, x_i\}$ in the boundary of $\partial B_i$ ($i=1,2$), and identify $x_1$ with $x_2$. 
The resulting $3$-complex $W$ is a wedge of non-evasive $3$-balls, hence non-evasive. To obtain the requested $3$-ball, attach onto $W$ a square pyramid with apex $x_{1} \equiv x_2$ and basis $\{a_1, a_2, b_1, b_2\}$; then, subdivide the pyramid stellarly, by inserting a vertex $v$ in the barycenter of the quadrilateral $\{a_1, a_2, b_1, b_2\}$. Let $B$ be the resulting triangulation. By construction, $\Lk(v,B)$ is non-evasive. Since the deletion of $v$ from $B$ is $W$, which is non-evasive, we conclude that $B$ is non-evasive.
\end{proof}

\begin{theorem} \label{thm:PiTightNonCollapsible}
If $K$ is a composite knot of bridge index $3$, or any knot of bridge index $\le 2$, there exists a triangulated $3$-ball $B = B(K)$ with the following properties:
\begin{compactenum}[\rm (i)]
\item $B$ is non-evasive;
\item $B$ contains a $6$-edge subcomplex isotopic to $K$;
\item $B$ does not admit {\em any} linear embedding in $\mathbb{R}^3$.
\end{compactenum}
\end{theorem}

\begin{proof}
Some perfect triangulation $S_0$ of the $3$-sphere contains $K$ as $3$-edge subcomplex: This follows from Proposition~\ref{prop:LickorishClaim} in case $K$ is a composite knot of bridge index $3$, and from Corollary~\ref{cor:lickmar} if $K$ has bridge index~$\le 2$. Let $B_0$ be the $3$-ball $S_0 - \Delta$. This $B_0$ is collapsible, yet it contains a $3$-edge knot isotopic to $K$. Since the barycentric subdivision of any collapsible complex is non-evasive, $B_1 = B_1(K) := \sd B_0$ is a non-evasive $3$-ball with a $6$-edge knot isotopic to $K$ in its $1$-skeleton. If $B_1 (K)$ has a linear embedding in $\mathbb{R}^3$, then its subcomplex $K$ has also a linear embedding in $\mathbb{R}^3$, as a closed path consisting of $6$ straight edges. Now:
\begin{compactitem}[ --- ]
\item if $K$ has stick number $>6$, set $B:= B_1 (K)$;
\item if $K$ has stick number $\le 6$, choose a knot $K'$ of stick number $>6$, form the ball $B_1(K')$ as above, and then thicken the wedge of $B_1(K)$ and $B_1(K')$ to a non-evasive $3$-ball~$B$, as in Lemma~\ref{lem:glueNE}.
\end{compactitem}
In both cases, $B$ contains at least one knot of stick number $>6$ as $6$-edge subcomplex, so it cannot have a linear embedding in $\R^3$. 
\end{proof}

\begin{Cor} \label{cor:NEnotConvex}
For each $d \ge 3$, some non-evasive $d$-balls are not convex.
\end{Cor}

\begin{proof}
For $d=3$, the claim follows from Theorem~\ref{thm:PiTightNonCollapsible} together with the fact that all convex $3$-balls can be linearly embedded in $\mathbb{R}^3$.

For $d \ge 4$, let us start by exhibiting a $(d-1)$-sphere $S$ whose barycentric subdivision is not shellable. By Proposition~\ref{prop:Bing}, some $3$-sphere $S_0$ has a $3$-edge subcomplex $K_0$ isotopic to a connected sum of $3 \cdot 2^{d-4} \cdot (d-2)!$ trefoil knots. The $(d-4)$-th suspension $S$ of $S_0$ is a PL $(d-1)$-sphere, because suspensions preserve the PL property and all $3$-spheres are PL. Inside $S$, the $(d-4)$-th suspension $K$ of $K_0$ has exactly $3 \cdot 2^{d-4}$ facets. So, inside $\sd S$, $\sd K$ has $3 \cdot 2^{d-4} \cdot (d-2)!$ facets. Moreover, the space $|\sd S| - |\sd K| = |S| - |K|$ retracts to $|S_0| - |K_0|$. Reasoning as in~\cite[Corollary~2.21]{BZ}, one can prove that $\sd S$ is (not LC and therefore) not shellable.

Now, Pachner showed that every PL $(d-1)$-sphere is combinatorially equivalent to the boundary of some shellable $d$-ball~\cite[Theorem~2, p.~79]{Pachner0}. In particular, there is a shellable $d$-ball $B$ whose boundary is our $(d-1)$-sphere $S$. But then the boundary of $\sd B$ is $\sd S$, which is not shellable. Since the boundary of every convex polytope is shellable, $\sd B$ cannot admit convex embeddings. Since $B$ is shellable, and hence collapsible, $\sd B$ is non-evasive.
\end{proof}

\begin{rem}
For each $d \ge 4$ and for any non-negative integer $k$, one can even construct vertex decomposable $d$-balls whose $k$-th barycentric subdivision is not convex. To see this, 
use the same proof of Corollary~\ref{cor:NEnotConvex}, starting with a $3$-sphere $S_0$ with a $3$-edge knot $K$ isotopic to a connected sum of $3 \cdot 2^{d-4} \cdot \left(\, (d-2)! \,\right)^k$ trefoil knots. 
\end{rem}

In conclusion, knotted subcomplexes of $3$-balls yield obstructions both to collapsibility and embeddability in $\R^3$. While the first obstructions directly depends on the \emph{bridge index} of the knot, the second obstruction is related to the \emph{stick number} of the knot. The existence of knots with high stick number and low bridge index makes it possible to construct $3$-balls that are collapsible (and even non-evasive or shellable), but cannot be linearly embedded in $\R^3$.

\section{Higher dimensions and open problems} \label{sec:HighDim}
In dimension higher than $3$, we do not know much. First of all, let us show that Theorem~\ref{thm:tomo3ball} does not generalize to higher dimensions: In fact, a ball that is  $\pi$-tight just with respect to \emph{some} directions $\pi$, need not be collapsible. 

\begin{theorem}\label{thm:tomo4ball} 
Let $\pi$ be an arbitrary vector in $\mathbb{R}^4$.
\begin{compactenum}[\rm (1)]
\item For each $m \in \mathbb{N}$, there is a $\pi$-tight $4$-ball whose $m$-th barycentric subdivision is not collapsible.
\item There is a $\pi$-tight contractible $3$-complex in $\mathbb{R}^4$ any subdivision of which is not collapsible.
\end{compactenum}
\end{theorem}

\begin{proof}
Up to rotating the coordinate axes, we can assume that $\pi$ is the vector $(0,0,0,1)$. 

By \cite[Theorem~3.12]{KarimBrunoZeeman}, for any $m \in \mathbb{N}$ there is a $3$-ball $B$ such that $\sd^m (B \times I)$ is not collapsible. The proof of \cite[Theorem~3.12]{KarimBrunoZeeman} runs as follows: For {any} $3$-ball $B$ that contains a knotted spanning edge, if the knot is the sum of sufficiently many trefoil knots, then $\sd^m (B \times I)$ is not collapsible. By Proposition~\ref{prop:Bing}, some $3$-balls $B$ of this type can be linearly embedded in~$\mathbb{R}^3$. Let us choose one. Since $\pi = (0,0,0,1)$, the product $B \times I$ is obviously $\pi$-tight. Thus $B \times I$ is the $4$-ball we are looking for.
As for claim (2), consider any contractible $2$-complex $C$ that is geometrically realizable in $\mathbb{R}^3$, but has no free face. Possible choices for $C$ include Bing's house with two rooms, or the $8$-vertex triangulation of the Dunce Hat~\cite{BL-duncehat}. Contractibility and the lack of free faces are preserved under suspensions, and under subdivisions as well. So,
\begin{compactitem}[ --- ]
\item topologically,  (any subdivision of) the suspension of $C$ is contractible;
\item combinatorially, (any subdivision of) the suspension of $C$ is not collapsible;
\item geometrically, if we use the vertices $(0,0,0,1)$ and $(0,0,0,-1)$ as apices,  (any subdivision of) the suspension of~$C$ is $\pi$-tight in $\mathbb{R}^4$.
\end{compactitem}
\vskip-5mm
\end{proof}

\begin{rem}
Theorem~\ref{thm:tomo4ball}, part (i), would follow immediately from Goodrick's claim~\cite{Goodrick1} that the suspension of any non-collapsible ball is non-collapsible. Unfortunately, the proof in~\cite{Goodrick1} is erroneous, as pointed out by Rushing~\cite{Rush}. However, using short knots in $3$-balls as in the proof of Corollary~\ref{cor:NEnotConvex}, one can see that the suspensions of \emph{some} non-collapsible balls are non-collapsible.
\end{rem}

\begin{rem}
Recently~\cite{KarimBrunoZeeman}, we showed that for every PL $4$-ball $B$ there is an integer $m$ (depending on $B$) such that $\sd^m (B \times I)$ is collapsible. So, while the complex $C \times I$ of Theorem~\ref{thm:tomo4ball} has no collapsible subdivision, the ball $B \times I$ does.
\end{rem}

Of course, convex balls are $\pi$-tight with respect to \emph{all} directions $\pi$.
So Theorem \ref{thm:tomo4ball} leaves the door open to the possibility that convex $d$-balls are all collapsible.  This is in fact an open problem, already for $d = 4$.

\begin{conjecture}[{Lickorish, cf.~Kirby~\cite[Problem~5.5]{Kirby}}]
All linear subdivisions of the $d$-simplex are collapsible.
\end{conjecture}
 
Goodrick extended the conjecture to ``all simplicial complexes with a star-shaped geometric realization, are collapsible'', cf.~\cite[Problem~5.5]{Kirby}. 
Recently, we proved a weaker version of Lickorish' and Goodrick's conjectures:

\begin{theorem}[{Adiprasito--Benedetti~\cite[Theorem~3.42]{KarimBrunoMG&C}}] \label{1}
Every star-shaped $d$-ball becomes collapsible after at most $d-2$ barycentric subdivisions. 
\end{theorem}

It is natural to ask whether the previous result can be extended to different topologies. 

\begin{question} \label{question:question}
Does every tight $d$-manifold become perfect after $d-2$ barycentric subdivisions? 
\end{question}

We know the answer is positive if the manifold is acyclic, by Theorem~\ref{1}.   
Question~\ref{question:question} admits positive answer also for homology spheres, by combining \cite[Theorem 3.42]{KarimBrunoMG&C} with \cite[Corollary~3.6]{Kuehnel}. Unfortunately, our proof of Theorem~\ref{1} does not extend to different topologies, and our proofs of Theorems~\ref{thm:tomo3ball} and \ref{thm:tomo3man} do not extend to higher dimensions. 

\medskip 
Finally, our knot-theoretical approach of Section~\ref{sec:knots} suggests the existence of a linear relation between the bridge index and the minimal number of critical faces in a discrete Morse function. For brevity, let us say that a $d$-complex ``has discrete Morse vector $(c_0, \ldots, c_d)$'' if it admits some discrete Morse function with exactly $c_i$ critical faces of dimension~$i$. For example, being collapsible is the same as having discrete Morse vector $(1, 0, \ldots, 0)$.

\begin{conjecture} \label{conj:knot1}
Any knot of bridge index $b$ appears: 
\begin{compactenum} [$\!\!\!\!\!$\rm(i)]
\item as knotted spanning edge in some $3$-ball with discrete Morse vector $(1,b-2,b-2,0)$,  if $b \ge 2$;
\item as $3$-edge subcomplex of some $3$-sphere with discrete Morse vector $(1,b-3,b-3,1)$, if $b \ge 3$.
\end{compactenum}
\end{conjecture}

Conjecture~\ref{conj:knot1}-(i) holds true for $b=2$ (Proposition~\ref{prop:LickMar}).
Conjecture~\ref{conj:knot1}-(ii) holds true for $b=3$, at least if the knot is composite (Proposition~\ref{prop:LickorishClaim}). More generally, if one could prove item (i) for all knots of bridge index $b \le b_{\max}$, then one could use the ideas of Proposition~\ref{prop:LickorishClaim} to prove item (ii) for all \emph{composite} knots of bridge index $b \le b_{\max} + 1$.

Conjecture~\ref{conj:knot1} is ``best possible''. In fact, if $b\ge 2$ and $T_{b-1}$ is a connected sum of $b-1$ trefoil knots, then the bridge index of $T_{b-1}$ is $b$; in this case,
\begin{compactenum} [(i)]
\item  for $b=3$, Goodrick~\cite{GOO} showed that $T_{b-1}$ cannot appear as knotted spanning edge in any $3$-ball with discrete Morse vector $(1,b-3,b-3,0)$.
\item for any $b \ge 4$, the second author~\cite{Benedetti-DMT4MWB} showed that  $T_{b-1}$ cannot appear as $3$-edge subcomplex in any sphere with discrete Morse vector $(1,b-4,b-4,1)$ or smaller. 
\end{compactenum}

\section*{Acknowledgments}
Thanks to G\"{u}nter Ziegler, Frank Lutz and Igor Pak and for useful suggestions and references.

\begin{small}

\end{small}

\begin{thebibliography}{10}
\itemsep=-1.4mm

\bibitem{KarimBrunoMG&C}
{\sc K.~Adiprasito and B.~Benedetti}, Metric geometry and collapsibility. Preprint (2012) at 
\url{arxiv:1107.5789}.

\bibitem{KarimBrunoZeeman}
{\sc K.~Adiprasito and B.~Benedetti}, Subdivision, shellability, and the Zeeman conjecture. Preprint (2012) at \url{arxiv:1202.6606}.

\bibitem{Aumann}
{\sc G.~Aumann}, On a topological characterization of compact convex point sets. Annals of Math. 37 (1936), 443--447.

\bibitem{Benedetti-DMT4MWB} 
{\sc B.~Benedetti}, Discrete Morse theory for manifolds with boundary. Trans. Amer. Math. Soc., to appear. Preprint at \url{arXiv:1007.3175}.

\bibitem{Benedetti-DMTasp} 
{\sc B.~Benedetti}, Discrete Morse theory is at least as perfect as  Morse Theory. Preprint at \url{arXiv:1010.0548}.

\bibitem{BenedettiOWR}
{\sc B.~Benedetti}, Non-evasiveness, collapsibility, and explicit knotted triangulations. Oberwolfach Reports 8, Issue 1 (2011), 403--405.

\bibitem{BenedettiLutz}
{\sc B.~Benedetti and F.~H.~Lutz}, Knots in collapsible and non-collapsible balls. In preparation (2012).


\bibitem{BL-duncehat}
{\sc B.~Benedetti and F.~H.~Lutz}, The dunce hat and a minimal non-extendably collapsible 3-ball. EG Models, to appear. Preprint (2009) at \url{arXiv:0912.3723}.




\bibitem{BZ}
{\sc B.~Benedetti and G.~M. Ziegler}, On locally constructible spheres and
  balls. Acta Mathematica 206 (2011), 205--243. 

\bibitem{BING}
{\sc R.~H.~Bing}, Some aspects of the topology of 3-manifolds related to
  the {P}oincar\'e conjecture. In Lectures on Modern Mathematics (T.~Saaty
  ed.), vol.~II, Wiley, 1964, 93--128.


\bibitem{BruggesserMani}
{\sc H.~Bruggesser and P.~Mani}, Shellable decompositions of cells and
  spheres. Math.~Scand. 29 (1971), 197--205.

\bibitem{Calvo}
{\sc J.~A.~Calvo}, Geometric knot spaces and polygonal isotopy. J.~Knot Th.~Ram. 10 (2001), 245--267.

\bibitem{CHIL}
{\sc D.~R.~J. Chillingworth}, Collapsing three-dimensional convex
  polyhedra. Proc. Camb. Phil. Soc. 63 (1967), 353--357. Erratum in 88 (1980), 307--310.

\bibitem{Effenberger}
{\sc F.~Effenberger}, Stacked polytopes and tight triangulations of manifolds. J. Comb. Theory Ser. A, to appear. Preprint at \url{arXiv:0911.5037}.


\bibitem{EH}
{\sc R.~Ehrenborg and M.~Hachimori}, Non-constructible complexes and the bridge index. Europ. J. Comb. 22 (2001), 475--491.

\bibitem{FormanADV}
{\sc R.~Forman}, {M}orse theory for cell complexes. Adv.~in Math. 134 (1998), 90--145.


\bibitem{Furch}
{\sc R.~Furch},  Zur {G}rundlegung der kombinatorischen {T}opologie.  Abh.
  Math. Sem. Univ. Hamburg 3 (1924), 69--88.


\bibitem{Goodrick1}
{\sc R.~E.~Goodrick} [credited as R.~Goordick], A note on simplicial collapsing. Yokohama Math. Journal 15 (1967), 33--34.

\bibitem{GOO}
{\sc R.~E.~Goodrick}, Non-simplicially collapsible triangulations of
  {$I^n$}. Proc. Camb. Phil. Soc. 64 (1968), 31--36.

\bibitem{HachiZiegler}
{\sc M.~Hachimori and G.~Ziegler}, Decompositions of balls and spheres with knots consisting of few edges.
Mathematische Zeitschrift, 235-1 (2000), 159--171.

\bibitem{HamstromJerrard}
{\sc M.~E.~Hamstrom and R.~P.~Jerrard}, Collapsing a triangulation of a ``knotted'' cell. Proc. Amer. Math. Soc. 21 (1969), 327--331. 

\bibitem{KahnSaksSturtevant}
{\sc J.~Kahn, M.~Saks and D.~Sturtevant}, A topological approach to evasiveness. Combinatorica 4 (1984), 297--306.

\bibitem{KAWA}
{\sc A.~Kawauchi}, A survey of knot theory. Birkh\"{a}user, Basel, 1996.

\bibitem{Kirby}
{\sc R.~C.~Kirby}, Problems in low-dimensional topology (1995). Available online at \url{http://math.berkeley.edu/~kirby/}. 

\bibitem{KuiperA}
{\sc N.~H.~Kuiper}, Minimal total absolute curvature for immersions. Invent. Math. 10 (1970) 209--238. 

\bibitem{KuiperB}
{\sc N.~H.~Kuiper}, Morse relations for curvature and tightness. In Proc. Liverpool Singularities Symp. II (C.~T.~C.~Wall ed.), Springer, Lecture Notes in Mathematics 1612, 1995. 

\bibitem{Kuehnel}
{\sc W.~K\"uhnel}, Tight polyhedral submanifolds and tight triangulations. Springer, Lecture Notes in Mathematics 209, 1971. 

\bibitem{LICK}
{\sc W.~B.~R. Lickorish}, Unshellable triangulations of spheres. Europ. J. Combin. 12 (1991), 527--530.

\bibitem{LICKMAR}
{\sc W.~B.~R. Lickorish and J.~M. Martin}, Triangulations of the 3-ball
  with knotted spanning $1$-simplexes and collapsible $r$-th derived
  subdivisions. Trans. Amer. Math. Soc. 170 (1972), 451--458.

\bibitem{Pachner0}
{\sc U.~Pachner}, Konstruktionsmethoden und das kombinatorische
  {H}om\"{o}omorphieproblem f\"{u}r {T}riangulationen kompakter semilinearer
 {M}annigfaltigkeiten. Abh. Math. Sem. Hamb. 57 (1987), 69--86.
 


\bibitem{Rolfsen}
{\sc D.~Rolfsen}, Knots and Links. Publish or Perish, 1976. See also \url{http://katlas.org/wiki/The_Rolfsen_Knot_Table}.

\bibitem{Rush}
{\sc T.~B.~Rushing}, Topological embeddings. Academic Press, 1973.

\bibitem{Welker}
{\sc V.~Welker}, Constructions preserving evasiveness and collapsibility. Discr. Math. 207 (1999), 243--255. 

\bibitem{Whitehead}
{\sc J.~H.~C. Whitehead}, Simplicial spaces, nuclei and {$m$}-groups.
  Proc. Lond. Math. Soc., 45 (1939), 243--327.

\bibitem{Z}
{\sc G.~M.~Ziegler}, Lectures on polytopes. Springer, Graduate Texts in Math. 152, 1998.

\end{thebibliography}
\end{document}